\documentclass{amsart}

\usepackage{amssymb,amsthm,epsfig,color,xfrac}
\usepackage{thm-restate}

\usepackage[colorlinks=true, linkcolor=blue, citecolor=green]{hyperref}
\usepackage{MnSymbol}

\usepackage{pinlabel}
\usepackage{palatino}

\usepackage[all]{xy}

\newcommand{\mc}[1]{\mathcal{#1}}

\newcommand{\co}{\colon\thinspace}

\usepackage{amsmath}
\newcommand{\rightQ}[2]{\left.\raisebox{.2em}{$#1$}\middle/\raisebox{-.2em}{$#2$}\right.}
\newcommand{\leftQ}[2]{\left.\raisebox{-.2em}{$#2$}\middle\backslash\raisebox{.2em}{$#1$}\right.}

\def\Stab {\mathrm{Stab}}

\newtheorem{theorem}{Theorem}[section]

\newtheorem{proposition}[theorem]{Proposition}

\newtheorem*{claim*}{Claim}

\theoremstyle{definition}

\newtheorem{definition}[theorem]{Definition}

\begin{document}

\title{Relative Cubulations and groups with a $2$--sphere boundary}

\author[D. Groves]{Daniel Groves}
\author[E. Einstein]{Eduard Einstein}
\address{Department of Mathematics, Statistics, and Computer Science,
University of Illinois at Chicago,
322 Science and Engineering Offices (M/C 249),
851 S. Morgan St.,
Chicago, IL 60607-7045}
\email{groves@math.uic.edu}
\email{einstein@uic.edu}

\begin{abstract}
We introduce a new kind of action of a relatively hyperbolic group on a CAT$(0)$ cube complex, called a relatively geometric action.  We provide an application to characterize finite-volume Kleinian groups in terms of action on cube complexes, analogous to the results of Markovic and Ha\"issinsky in the closed case.
\end{abstract}
\thanks{The second author was partially supported by the National Science Foundation, DMS-1507067.}

\maketitle

\section{Introduction}

The Cannon Conjecture (see \cite[Conjecture 11.34]{Cannon91}, \cite[Conjecture 5.1]{CannonSwenson}) is one of the central problems in geometric group theory. Using the work of Agol \cite{VH}, Markovic \cite[Theorem 1.1]{Markovic:Cannon} gave an approach to proving the Cannon Conjecture using CAT$(0)$ cube complexes and quasi-convex codimension $1$ surface subgroups.  This was slightly generalized by Ha\"{i}ssinsky \cite[Theorem 1.10]{Hais:Cannon}, who proved that a Gromov hyperbolic group whose boundary is a $2$--sphere is virtually Kleinian if and only if it acts properly and cocompactly on a CAT$(0)$ cube complex\footnote{In fact, Ha\"{i}ssinsky proved this result more generally for hyperbolic groups with planar boundary}.

A relative version of the Cannon Conjecture states that a relatively hyperbolic group (with abelian parabolic subgroups) whose (Bowditch) boundary is a $2$--sphere is a Kleinian group (see \cite[Problem 57]{kapovichproblems}, for example).  In \cite{GMS}, it was proved that the Relative Cannon Conjecture is implied by the Cannon Conjecture.

In this paper we introduce a new kind of action of a relatively hyperbolic group on a CAT$(0)$ cube complex, called a {\em relatively geometric} action (see Definition \ref{def:RG} below).  Contrary to proper and cocompact actions, whose coarse geometry is that of the Cayley graph, relatively geometric actions have the coarse geometry of the coned Cayley graph (see Proposition \ref{p:coned} below), and hence their geometry can be expected to more faithfully exhibit the relatively hyperbolic geometry of groups.

It follows from the work of Cooper--Futer \cite[Theorem 1.1]{CooperFuter} and the Sageev construction \cite{Sageev-Codim1} that if $M$ is a finite-volume hyperbolic $3$--manifold then $\pi_1(M)$ admits a relatively geometric action on a CAT$(0)$ cube complex (see Theorem \ref{t:M RG} below).

Applying the results from \cite{GMS}, \cite{Omnibus} and \cite{CooperFuter}, we prove the following relative version of Ha\"{i}ssinsky's result.

\begin{restatable}{theorem}{relhai} \label{t:RelHai}
Suppose that $(G,\mc{P})$ is relatively hyperbolic, that the elements of $\mc{P}$ are free abelian, and that the (Bowditch) boundary of $(G,\mc{P})$ is a $2$--sphere.  Then $G$ is Kleinian if and only if   $G$ admits a relatively geometric action on a CAT$(0)$ cube complex.
\end{restatable}

\section{Relative cubulations}

The theory of hyperbolic groups acting properly and cocompactly on CAT$(0)$ cube complexes is by now well developed (see \cite{Wise,HaglundWise12,HsuWise15,AgolGrovesManning-alternateMSQT,Omnibus}, etc.).  In the relatively hyperbolic situation, two generalizations have been previously studied: proper and cocompact actions (as in \cite{SageevWise15,Wise}) and proper and `cosparse' actions (see \cite{hruskawise:finiteness,SageevWise15}).

The following definition provides another condition for relatively hyperbolic groups which restricts to being proper and cocompact in case $G$ is hyperbolic and $\mc{P} = \emptyset$.

\begin{definition} \label{def:RG}
 Suppose that $(G,\mc{P})$ is a group pair.  A (cellular) action of $G$ on a cell complex $X$ is {\em relatively geometric (with respect to $\mc{P}$)} if 
\begin{enumerate}
\item  $\leftQ{X}{G}$ is compact;
\item Each element of $\mc{P}$ acts elliptically on $X$; and
\item Each stabilizer in $G$ of a cell in $X$ is either finite or else conjugate to a finite-index subgroup of $\mc{P}$.
\end{enumerate}
\end{definition}

In Section \ref{s:man} below we give natural examples of relatively geometric actions on CAT$(0)$ cube complexes, provided by the work of Cooper and Futer \cite{CooperFuter}. 

The authors will investigate relatively geometric actions of relatively hyperbolic groups on CAT$(0)$ cube complexes in future work.  For the remainder of this section, we record some basic features of relatively geometric actions.

The following is an immediate consequence of \cite[Theorem 5.1]{CharneyCrisp}.
\begin{proposition} \label{p:coned}
Suppose that $(G,\mc{P})$ is relatively hyperbolic, and that $G$ admits a relatively geometric action on a CAT$(0)$ cube complex $X$.  Then $X$ is quasi-isometric to the coned-off Cayley graph of $(G,\mc{P})$, and consequently is $\delta$--hyperbolic for some $\delta$.
\end{proposition}
 
The following is an immediate consequence of \cite[Corollary 6.5]{Omnibus}.  See \cite[$\S6$]{Omnibus} for the definition of $\mc{Q}$--fillings, and more context.
\begin{proposition} \label{prop:fill CAT(0)}
Suppose that $(G,\mc{P})$ is relatively hyperbolic and that $G$ admits a relatively geometric action on a CAT$(0)$ cube complex $X$.  Let $\mc{Q}$ be a collection of finite-index subgroups of elements of $\mc{P}$ so that any infinite cell stabilizer contains a conjugate of an element of $\mc{Q}$.  For sufficiently long $\mc{Q}$--fillings 
\[	G \to \overline{G} = \rightQ{G}{K}	\]
of $(G,\mc{P})$, the quotient $\leftQ{X}{K}$ is a CAT$(0)$ cube complex.
\end{proposition}
Proposition \ref{prop:fill CAT(0)} provides one major benefit of relatively geometric actions over proper and cocompact (or proper and cosparse) actions.  Namely, if the images of the elements of $\mc{P}$ in $\rightQ{G}{K}$ are hyperbolic and virtually special (for example, finite or virtually cyclic) then \cite[Theorem D]{Omnibus} implies that $\rightQ{G}{K}$ is virtually special.  This allows one to prove properties of $G$ by taking virtually special hyperbolic  Dehn fillings $\rightQ{G}{K}$ and applying the properties of virtually special hyperbolic groups.  This technique is used in the proof of Theorem \ref{t:RelHai} below, and will be crucial in our future work.

{\em Relatively quasi-convex subgroups} of relatively hyperbolic groups were investigated in \cite{HruskaQC}.  See that paper for many equivalent definitions, or \cite{agm,QC-DF} for yet more equivalent definitions.  A relatively quasi-convex subgroup $H$ of a relatively hyperbolic group $(G,\mc{P})$ is {\em full} if for any $g \in G$ and $P \in \mc{P}$, the subgroup $H^g \cap P$ is either finite, or of finite-index in $P$.

\begin{proposition} \label{p:P-ell}
Suppose that $(G,\mc{P})$ is relatively hyperbolic and that $\mc{H}$ is a finite collection of full relatively quasi-convex codimension $1$ subgroups.  Then each one-ended element of $\mc{P}$ acts elliptically on the cube complex dual to $\mc{H}$.
\end{proposition}
\begin{proof}
Let $\mc{H}$ be a finite collection of full relatively quasi-convex codimension $1$ subgroups of $G$ and let $X$ be a CAT$(0)$ cube complex  dual to $\mc{H}$ obtained by the Sageev construction.  In order to obtain a contradiction, suppose that $P \in \mc{P}$ is one-ended and that $P$ does not act elliptically on $X$.  Then any orbit $P . x$ of $P$ in $X$ is unbounded, and for any such orbit there is a hyperplane $W$ in $X$ so that there are elements of $P.x$ on either side of $W$, arbitrarily far from $W$.

It is straightforward to see that $\Stab(W) \cap P$ is a codimension $1$ subgroup of $P$.  Since $P$ is one-ended, every codimension $1$ subgroup of $P$ is infinite.  But $\Stab(W)$ is full so $\Stab(W) \cap P$ is finite-index in $P$.  It follows that any orbit $P . x$ is contained in a bounded neighborhood of $W$, contradicting our choice of $W$.
\end{proof}

For a hyperbolic group $G$, Sageev \cite[Theorem 3.1]{Sageev-Codim1} proved that the cube complex associated to a finite collection of quasi-convex codimension $1$ subgroups is $G$--cocompact.  The following is the appropriate relatively geometric version, and follows quickly from results of Hruska--Wise \cite{hruskawise:finiteness}.

\begin{proposition} \label{p:cocompact}
Suppose that $(G,\mc{P})$ is relatively hyperbolic, that each element of $\mc{P}$ is one-ended, and that $\mc{H}$ is a finite collection of full relatively quasi-convex codimension $1$ subgroups.  Then the action of $G$ on the cube complex dual to $\mc{H}$ is $G$--cocompact.
\end{proposition}
\begin{proof}
The condition that elements of $\mc{H}$ are full implies that the cubulation of each element of $\mc{P}$ induced by $\mc{H}$ (in any variation, see \cite{hruskawise:finiteness}) is a finite cube complex.  The result now follows immediately from \cite[Theorem 7.12]{hruskawise:finiteness}.
\end{proof}

The following result is a slight variation of \cite[Theorem 5.1]{bergeronwise}, and provides a useful criterion for actions to be relatively geometric.

\begin{theorem} \label{t:Rel BW}
Let $(G,\mc{P})$ be relatively hyperbolic and suppose that for every pair of distinct points $u,v \in \partial(G,\mc{P})$ there is a full relatively quasi-convex codimension $1$ subgroup $H$ of $G$ so that $u,v$ lie in $H$--distinct components of $\partial G \smallsetminus \Lambda H$.  Then there exist finitely many full relatively quasi-convex codimension $1$ subgroups of $G$ so that the action of $G$ on the dual cube complex is relatively geometric.
\end{theorem}
\begin{proof}
For {\em any} finite collection of full quasi-convex codimension $1$ subgroups, the action on the dual cube complex is $G$--cocompact by Proposition \ref{p:cocompact}, and elements of $\mc{P}$ act elliptically by Proposition \ref{p:P-ell}.  Therefore, it remains to prove that there is a finite collection of full relatively quasi-convex codimension $1$ subgroups with respect to which the stabilizers for the dual cube complex are finite or parabolic.
Thus, we need to show that there is a finite collection of such subgroups which `cut' each loxodromic element of $G$.  This can be achieved by applying the proof of \cite[Theorem 5.1]{bergeronwise} directly.
\end{proof}

\section{Finite-volume hyperbolic $3$--manifolds and relatively geometric actions}\label{s:man}
In this section, we explain how the works of Cooper--Futer \cite{CooperFuter} and Bergeron--Wise \cite{bergeronwise} together imply the following result.

\begin{theorem} \label{t:M RG}
Suppose that $M$ is a finite-volume hyperbolic $3$--manifold.  Then $\pi_1(M)$ admits a relatively geometric action on a cube complex.
\end{theorem}
In case $M$ is closed, this result is due to Bergeron--Wise \cite{bergeronwise}, using work of Kahn--Markovic \cite{KM12}.  In the finite-volume case, Cooper and Futer proved that $\pi_1(M)$ admits a proper and cocompact action on a CAT$(0)$ cube complex, using \cite[Theorem 1.2]{CooperFuter} and the results in \cite{bergeronwise}. 

\begin{definition} \cite{CooperFuter}
A collection of immersed surfaces in a hyperbolic $3$--manifold $M$ is {\em ubiquitous} if for any pair of hyperbolic planes $\Pi, \Pi' \subset \mathbb{H}^3$ whose distance $d(\Pi,\Pi')$ is greater than $0$ there is some surface $S$ in the collection with an embedded preimage $\widetilde{S} \subset \mathbb{H}^3$ that separates $\Pi$ from $\Pi'$.
\end{definition}

\begin{theorem} \cite[Theorem 1.1]{CooperFuter} \label{t:CF}
Let $M$ be a complete, finite-volume hyperbolic $3$--manifold.  Then the set of closed immersed quasi-Fuchsian surfaces in $M$ is ubiquitous.
\end{theorem}

Noting that the closed surfaces in Theorem \ref{t:CF} contain no parabolics, and so the corresponding subgroups of $\pi_1(M)$ are full, Theorem \ref{t:M RG} is an immediate consequence of Theorems \ref{t:CF} and \ref{t:Rel BW}.

\section{Criterion for Relative Cannon}
For the convenience of the reader, we recall the statement of Theorem \ref{t:RelHai}.

\relhai*
\begin{proof}
Suppose that $G$ is Kleinian.  Then $G$ admits a relatively geometric action by Theorem \ref{t:M RG} above.  

Conversely, suppose that $G$ admits a relatively geometric action on a CAT$(0)$ cube complex $X$.  Let $\mc{Q}$ be as in the statement of Proposition \ref{prop:fill CAT(0)}.  According to Proposition \ref{prop:fill CAT(0)}, for sufficiently long $\mc{Q}$--fillings $G \to \rightQ{G}{K}$ the quotient space $\leftQ{X}{K}$ is a CAT$(0)$ cube complex.
 
 We consider sufficiently long co-(virtually cyclic fillings), obtained by choosing cyclic subgroups as filling kernels.\footnote{Note that an element $P \in \mc{P}$ fixes a point $\xi_P \in \partial(G,\mc{P})$, and by the dynamical characterization of relatively hyperbolic groups \cite{Yaman} the group $P$ acts properly and cocompactly on $\partial(G,\mc{P}) \smallsetminus \xi_P \cong \mathbb{R}^2$.  It follows that each element of $\mc{P}$ is free abelian of rank $2$.}  According to \cite[Theorem 1.2]{GMS}, for sufficiently long such fillings the quotient $\rightQ{G}{K}$ is a word-hyperbolic group whose (Gromov) boundary is a $2$--sphere.  On the other hand, such a $\rightQ{G}{K}$ acts cocompactly on the CAT$(0)$ cube complex $\leftQ{X}{K}$, with virtually cyclic cell stabilizers.  Since virtually cyclic groups are virtually special, and virtually cyclic subgroups of hyperbolic groups are quasi-convex, it follows from \cite[Theorem D]{Omnibus} that such a $\rightQ{G}{K}$ is a virtually special group (and in particular it is cubulable).  By \cite[Theorem 1.10]{Hais:Cannon}, any such $\rightQ{G}{K}$ is virtually Kleinian.  In fact, since the parabolic subgroups of $G$ are free abelian, $G$ cannot have a finite normal subgroup.  It now follows from \cite[Theorem 7.2]{rhds} that such a $\rightQ{G}{K}$ has no finite normal subgroup, and hence it acts faithfully on its boundary.  Therefore, $\rightQ{G}{K}$ is Kleinian.
 
We now take a longer and longer sequence of fillings of this form, obtaining a collection of hyperbolic quotients $G \twoheadrightarrow G_i = \rightQ{G}{K_i}$ so that each $G_i$ is Kleinian.  As in the proof of \cite[Corollary 1.4]{GMS} we get a sequence of representations $\rho_i \co G \to \mathrm{Isom}(\mathbb{H}^3)$, and exactly as in \cite{GMS} this sequence must converge to a discrete faithful representation of $G$ into $\mathrm{Isom}(\mathbb{H}^3)$, which shows that $G$ is Kleinian, as required.  
\end{proof}

\bibliographystyle{abbrv}

\begin{thebibliography}{10}

\bibitem{VH}
I.~Agol.
\newblock The virtual {H}aken conjecture.
\newblock {\em Doc. Math.}, 18:1045--1087, 2013.
\newblock With an appendix by Agol, Daniel Groves, and Jason Manning.

\bibitem{agm}
I.~Agol, D.~Groves, and J.~F. Manning.
\newblock Residual finiteness, {QCERF} and fillings of hyperbolic groups.
\newblock {\em Geom. Topol.}, 13(2):1043--1073, 2009.

\bibitem{AgolGrovesManning-alternateMSQT}
I.~Agol, D.~Groves, and J.~F. Manning.
\newblock An alternate proof of {W}ise's {M}alnormal {S}pecial {Q}uotient
  {T}heorem.
\newblock {\em Forum of Mathematics, Pi}, 4, 2016.

\bibitem{bergeronwise}
N.~Bergeron and D.~T. Wise.
\newblock A boundary criterion for cubulation.
\newblock {\em Amer. J. Math.}, 134(3):843--859, 2012.

\bibitem{Cannon91}
J.~W. Cannon.
\newblock The theory of negatively curved spaces and groups.
\newblock In {\em Ergodic theory, symbolic dynamics, and hyperbolic spaces
  ({T}rieste, 1989)}, Oxford Sci. Publ., pages 315--369. Oxford Univ. Press,
  New York, 1991.

\bibitem{CannonSwenson}
J.~W. Cannon and E.~L. Swenson.
\newblock Recognizing constant curvature discrete groups in dimension {$3$}.
\newblock {\em Trans. Amer. Math. Soc.}, 350(2):809--849, 1998.

\bibitem{CharneyCrisp}
R.~Charney and J.~Crisp.
\newblock Relative hyperbolicity and {A}rtin groups.
\newblock {\em Geom. Dedicata}, 129:1--13, 2007.

\bibitem{CooperFuter}
D.~Cooper and D.~Futer.
\newblock Ubiquitous quasi-fuchsian surfaces in cusped hyperbolic 3-manifolds.
\newblock arXiv.org/abs/1705.02890, 2017.

\bibitem{rhds}
D.~Groves and J.~F. Manning.
\newblock {Dehn filling in relatively hyperbolic groups}.
\newblock {\em Israel Journal of Mathematics}, 168:317--429, 2008.

\bibitem{QC-DF}
D.~Groves and J.~F. Manning.
\newblock Quasiconvexity and {D}ehn filling.
\newblock arXiv.org/abs/1708.07968, 2017.

\bibitem{Omnibus}
D.~Groves and J.~F. Manning.
\newblock Hyperbolic groups acting improperly.
\newblock arxiv.org/abs/1808.02325, 2018.

\bibitem{GMS}
D.~Groves, J.~F. Manning, and A.~Sisto.
\newblock Boundaries of {D}ehn fillings.
\newblock arxiv.org/abs/1612.03497, 2016.

\bibitem{HaglundWise12}
F.~Haglund and D.~T. Wise.
\newblock A combination theorem for special cube complexes.
\newblock {\em Ann. of Math. (2)}, 176(3):1427--1482, 2012.

\bibitem{Hais:Cannon}
P.~Ha\"{i}ssinsky.
\newblock Hyperbolic groups with planar boundaries.
\newblock {\em Invent. Math.}, 201(1):239--307, 2015.

\bibitem{HruskaQC}
G.~C. Hruska.
\newblock Relative hyperbolicity and relative quasiconvexity for countable
  groups.
\newblock {\em Algebr. Geom. Topol.}, 10(3):1807--1856, 2010.

\bibitem{hruskawise:finiteness}
G.~C. Hruska and D.~T. Wise.
\newblock Finiteness properties of cubulated groups.
\newblock {\em Compos. Math.}, 150(3):453--506, 2014.

\bibitem{HsuWise15}
T.~Hsu and D.~T. Wise.
\newblock Cubulating malnormal amalgams.
\newblock {\em Invent. Math.}, 199(2):293--331, 2015.

\bibitem{KM12}
J.~Kahn and V.~Markovic.
\newblock Immersing almost geodesic surfaces in a closed hyperbolic three
  manifold.
\newblock {\em Ann. of Math. (2)}, 175(3):1127--1190, 2012.

\bibitem{kapovichproblems}
M.~Kapovich.
\newblock Problems on boundaries of groups and {K}leinian groups.
\newblock \url{https://www.math.ucdavis.edu/~kapovich/EPR/problems.pdf}, 2007.

\bibitem{Markovic:Cannon}
V.~Markovic.
\newblock Criterion for {C}annon's conjecture.
\newblock {\em Geom. Funct. Anal.}, 23(3):1035--1061, 2013.

\bibitem{Sageev-Codim1}
M.~Sageev.
\newblock Codimension-{$1$} subgroups and splittings of groups.
\newblock {\em J. Algebra}, 189(2):377--389, 1997.

\bibitem{SageevWise15}
M.~Sageev and D.~T. Wise.
\newblock Cores for quasiconvex actions.
\newblock {\em Proc. Amer. Math. Soc.}, 143(7):2731--2741, 2015.

\bibitem{Wise}
D.~T. Wise.
\newblock The structure of groups with a quasiconvex hierarchy, 2012.
\newblock Unpublished manuscript.

\bibitem{Yaman}
A.~Yaman.
\newblock A topological characterisation of relatively hyperbolic groups.
\newblock {\em J. Reine Angew. Math.}, 566:41--89, 2004.

\end{thebibliography}

\end{document}